\documentclass[12pt,a4paper]{amsart}
\usepackage{graphics}
\usepackage{epsfig}
\usepackage{graphicx}
\theoremstyle{plain}
\usepackage{amssymb}
\usepackage[ukrainian,english]{babel}
\advance\hoffset-20mm \advance\textwidth40mm

\newtheorem{theorem}{Theorem}
\newtheorem{lemma}{Lemma}
\newtheorem*{theo*}{Theorem}

\theoremstyle{definition}

\newtheorem*{definition*}{Definition}

\newtheorem{remark}{Remark}

\DeclareMathOperator{\Ker}{Ker}

\DeclareMathOperator{\rk}{rk}
\DeclareMathOperator{\ad}{ad}

\DeclareMathOperator{\Frac}{Frac}
\DeclareMathOperator{\Der}{Der}
\begin{document}
\sloppy
\title[Nilpotent Lie algebras of derivations with the center of small corank ]
{Nilpotent Lie algebras of derivations with the center of small corank}
\author
{Ie.Yu.Chapovskyi, L.Z.Mashchenko, A.P.Petravchuk}
\address{Ie.Yu.Chapovskyi: Department of Algebra and Mathematical Logic, Faculty of Mechanics and Mathematics,
Taras Shevchenko National University of Kyiv, 64, Volodymyrska street, 01033  Kyiv, Ukraine}
\email{safemacc@gmail.com}
\address{L.Z.Mashchenko: Kyiv National University of Trade and Economics, 19, Kioto street, Kyiv,  Ukraine}
\email{mashchenkoliudmila@gmail.com}
\address{A.P.Petravchuk:
Department of Algebra and Mathematical Logic, Faculty of Mechanics and Mathematics,
Taras Shevchenko National University of Kyiv, 64, Volodymyrska street, 01033  Kyiv, Ukraine}
\email{ apetrav@gmail.com , apetravchuk@univ.kiev.ua }
\date{\today}

%
\begin{abstract}
Let $\mathbb K$ be a field of characteristic zero, $A$ an integral domain over $\mathbb K$ with the field of fractions $R = \Frac(A),$ and $\Der_{\mathbb K}A$ the Lie algebra of all $\mathbb K$-derivations on $A$.
 Let  $W(A):=R\Der_{\mathbb K} A$ and $L$ a nilpotent subalgebra of rank $n$ over $R$ of the Lie algebra  $W(A).$ We prove that  if the center $Z=Z(L)$ is
 of rank $\geq n-2$ over $R$ and  $F=F(L)$ is the field of constants for $L$ in $R,$ then the Lie algebra $FL$
 is contained in a locally nilpotent subalgebra of $ W(A)$ of rank $n$ over $R$ with a natural basis over the field $R.$  It is also  also proved that the  Lie algebra $FL$ can be isomorphically embedded (as an abstract Lie algebra) into the triangular Lie algebra $u_n(F)$ which was studied early by other authors.
 \end{abstract}
\maketitle

\section{Introduction}
Let $\mathbb K$ be a field of characteristic zero, $A$ an integral domain over $\mathbb K,$ and $R = \Frac(A)$ its field of fractions. Recall that a $\mathbb K$-derivation $D$ on $A$ is a $\mathbb K$-linear operator on the vector space $A$ satisfying the Leibniz's rule $D(ab) = D(a)b + aD(b)$
for any $a,b \in A$. The set $\Der_{\mathbb K} A$ of all
$\mathbb K$-derivations on $A$ is a Lie algebra over $\mathbb K$ with the
Lie bracket $[D_1,D_2] = D_1 D_2 - D_2 D_1$.
The Lie algebra $\Der_{\mathbb K} A$ can be isomorphically embedded
into the Lie algebra $\Der_{\mathbb K} R$ (any derivation $D$ on $A$
can be uniquely extended on $R$ by the rule
$D(a/b)=(D(a)b - aD(b))/b^2,$ $ a,b \in A$).
We denote by $W(A)$ the subalgebra $R \Der_{\mathbb K} A$ of the Lie algebra
$\Der_{\mathbb K} R$  (note that $W(A)$ and $\Der_{\mathbb K} R$ are Lie algebras over the field $\mathbb K$ but not over $R$). Nevertheless,
$W(A)$ and $\Der_{\mathbb K} R$ are vector spaces over the field $R$,
so one can define the rank $\rk_R L$ for any subalgebra L of the Lie algebra $W(A)$ by the rule: $\rk_R L = \dim_R RL$. Every subalgebra $L$ of the Lie algebra $W(A)$ determines its field of constants in $R$:
$$F = F(L):= \{r \in R \ | \ D(r) = 0  \ \
\mbox{for all} \  D \in L\}.$$ The product
$FL = \{ \sum \alpha_i D_i \ | \ \alpha_i \in F,\; D_i \in L \}$ is a Lie algebra
over the field $F$, this Lie algebra often has  simpler structure than $L$
itself (note that such an extension of the ground field preserves the main
properties of $L$ from the viewpoint of Lie theory).

We study nilpotent subalgebras $L \subseteq W(A)$ of rank $n \ge 3$ over $R$
with the center $Z = Z(L)$ of rank $\geq n-2$ over $R$, i.e. with the center of corank $\leq 2$ over $R$. We prove that $FL$ is contained in a locally nilpotent subalgebra of $W(A)$ with  a natural basis  over $R$ similar to the standard basis of the triangular Lie algebra $U_n(F)$ (Theorem \ref{classification}). As a consequence, we get  an isomorphic  embedding (as Lie algebras) of the Lie algebra $FL$ over $F$ into the triangular Lie algebra $u_n(F)$ over $F$ (Theorem \ref{embed}). These results generalize main results of the papers \cite{Sysak} and \cite{Petr}. Note that the problem of classifying finite dimensional Lie algebras from Theorem \ref{classification} up to isomorphism is wild (i.e., it contains the hopeless problem of classifying pairs of square matrices up to similarity, see \cite{BP}). Triangular Lie algebras were studied in \cite{Bav1}
and \cite{Bav2}, they are locally nilpotent but not nilpotent.

We use standard notations. The ground field $\mathbb K$ is arbitrary of characteristic zero. If $F$ is a subfield of a field $R$ and
$r_1,\dots,r_k \in R,$ then $F \langle r_1,\dots,r_k \rangle$ is the set of
all linear combinations of $r_i$ with coefficients in $F$, it is a subspace in the $F$-space $R$, for an infinite set $\{r_1, \ldots , r_k, \ldots \}$ we use the notation $F\langle \{r_i\}_{i=1}^{\infty}\rangle .$  The triangular subalgebra $u_n(\mathbb K)$ of the Lie
algebra $W_n(\mathbb K):=\Der_{\mathbb K} \mathbb K[x_1,\dots,x_n]$ consists
of all the derivations on $\mathbb K[x_1,\dots,x_n]$ of the form
$$
D = f_1(x_2,\dots,x_n) \frac{\partial}{\partial x_1} + \dots +
f_{n-1}(x_n) \frac{\partial}{\partial x_{n-1}} +
f_n \frac{\partial}{\partial x_1},
$$
where $f_i \in \mathbb K[x_{i+1},\dots,x_n]$, $f_n \in \mathbb K$.
If $D \in W(A),$ then $\Ker D$ denotes the field of constants for $D$ in $R,$
i.e., $\Ker D = \{r \in R \ | \ D(r) = 0\}$.

\section{Main properties of nilpotent subalgebras of W(A)}
We often use the next relations for derivations which are well known (see, for example  \cite{Nowicki}).

Let $D_1, D_2 \in W(A)$ and $a,b \in R$. Then

	1)  $[aD_1,bD_2] = ab[D_1,D_2] + aD_1(b)D_2 - bD_2(a)D_1;$

	2)  if $a,b \in \Ker D_1 \cap \Ker D_2$, then
		$[aD_1,bD_2] = ab[D_1,D_2].$

The next two lemmas contain some results about derivations and Lie algebras of derivations.
\begin{lemma}(\cite{MP1}, Lemma 2)
	Let $L$ be a subalgebra of the Lie algebra $\Der_{\mathbb K} R$
	and $F$ the field of constants for $L$ in $R$.
	Then $FL$ is a Lie algebra over $F,$ and if $L$ is abelian,
	nilpotent or solvable, then so is $FL$, respectively.
\end{lemma}

\begin{lemma} (\cite{MP1}, Proposition 1)
	Let $L$ be a nilpotent subalgebra of the Lie algebra $W(A)$ with
	$\rk_R L < \infty $ and $F = F(L)$ the field of constants for $L$ in $R$.
	Then:
	\begin{enumerate}
		\item $FL$ is finite dimensional over $F;$
		\item if $\rk_R L = 1$, then $L$ is abelian and $\dim_F FL=1;$
		\item if $\rk_R L = 2$, then $FL$ is either abelian with
		$\dim_F FL = 2$ or $FL$ is of the form
		$$
		FL = F \langle  D_2, D_1, aD_1,\dots,\frac{a^k}{k!}D_1 \rangle ,
		$$
		for some $D_1,D_2 \in FL$ and $a \in R$ such that
		$[D_1,D_2]=0, \ \ D_2(a) = 1,\; D_1(a) = 0$.
	\end{enumerate}
\end{lemma}

\begin{lemma}\label{central}
	Let $L$ be a nilpotent subalgebra of the Lie algebra $W(A)$	of rank
	$n $ over $R$ with the center $Z=Z(L)$ of rank $k$ over $R.$
	Then $I:= RZ\cap L$ is an abelian ideal of $L$  with $\rk _{R}I=k.$
\end{lemma}
\begin{proof}
	By  Lemma 4 from \cite{MP1}, $I$ is an ideal of the Lie algebra $L.$
	Let us show that $I$ is abelian.
	Let us choose an arbitrary basis $D_1,\dots,D_{k}$ of the center $Z$
	over $R$ (i.e., a maximal by inclusion linearly independent over $R$
	subset of $Z$).
	One can easy to see that $D_1,\dots,D_{k}$ is a basis of the ideal $I$ as well, so  we can write  for each element $D \in I$
	$$
	D = a_1 D_1 + \dots + a_{k} D_{k}
	$$
	for some $a_1,\dots,a_{k} \in R$.
	Since $D_j \in Z,\; j=1,\dots,k$ it holds
	\begin{equation} \label{centrality}
	[D_j,D] = [D_j,\sum_{i=1}^{k} a_i D_i] =
	\sum_{i=1}^{k} D_j(a_i)D_i = 0
	\end{equation}
	for $j=1,\dots,k$. The derivations  $D_1,\dots,D_n$ are linearly independent over the field $R,$ hence we obtain  from (\ref{centrality}) that
	$D_j(a_i) = 0, \; i,j = 1,\dots,k$.
	Therefore  we have for each element
	$\overline{D} = b_1 D_1 + \dots b_{k} D_{k}$ of the ideal $I$
	the next equalities:
	$$
	[D,\overline{D}] = [\sum_{i=1}^{k} a_i D_i, \sum_{j=1}^{k} b_j D_j] =
	\sum_{i,j=1}^{k} a_i b_j [D_i,D_j] = 0,
	$$
	since $D_i(b_j) = D_j(a_i)=0$ as mentioned above.
	The latter means  that $I$ is an abelian ideal. Besides, obviously
	$\rk_R I=k.$
\end{proof}

\begin{lemma}\label{elements}
	Let $L$ be a nilpotent subalgebra of the Lie algebra $W(A)$, $Z = Z(L)$ the center of L,  $I := RZ \cap L$ and $F$ the field of constants for $L$ in $R.$ If for some $D \in L$ it holds  $[D, FI]\subseteq FI, [D,FI] \ne 0,$ then
	there exist a basis $D_1,\dots,D_m$ of the ideal  $FI$ of the Lie algebra $FL$ over $R$
	and $a \in R$ such that $D(a) = 1,\; D_i(a) = 0,\; i=1,\dots,m.$ Besides, 	 each element $\overline{D} \in FI$ is of the form
	$\overline{D} = f_1(a)D_1 + \dots + f_m(a)D_m$ for some polynomials
	$f_i \in F_1[t]$, where $F_1$ is the field of constants for the subalgebra
	$L_1 = FI+FD$ in $R$.
\end{lemma}
\begin{proof}
	By Lemma \ref{central}, the intersection $I=RZ\cap L$ is an abelian ideal of the Lie algebra $L$ and
therefore $FI$ is an abelian ideal of the Lie algebra $FL$. Choose a basis $D_1, \ldots , D_m$ of $FI$ over the field $R$ in such a way that $D_1, \ldots , D_m\in Z.$  Then $FZ$ is the center of the Lie algebra $FL.$
Now take any basis $T_1, \ldots , T_s$ of the $F$-space $FI$ (note that the Lie algebra $FL$ is finite dimensional over the field $F$ by \cite{MP1}). Every basis element $T_i$ can be written in the form $T_i=\sum _{j=1}^{m}r_{ij}D_j, i=1, \ldots s$ for some $r_{ij} \in R.$ Denote by $B$ the subring $B=F[r_{ij}, i=1, \ldots s, j=1, \ldots ,m]$ of the field $R$ generated by $F$ and the elements $r_{ij}.$ Since the linear operator $\ad D$ is nilpotent on the $F$-space $FI$ the derivation $D$ is locally nilpotent on the ring $B.$ Indeed,
$$ [D, T_i]=[D, \sum _{j=1}^{m}r_{ij}D_j]=\sum _{j=1}^{m}D(r_{ij})D_j$$
and therefore
$$(\ad D)^{k_i}(T_i)=\sum _{j=1}^{m}D^{k_i}(r_{ij})D_j=0$$
for some natural $k_i, i=1, \ldots ,s.$ Denoting ${\overline k}=\max _{1\leq t\leq s}k_{t}$ we get $D^{\overline k}(r_{ij})=0$ and therefore $D$ is locally nilpotent on $B.$ One can easily show that there exists an element $p\in B$ (a preslice) such  that $D(p)\in \Ker D, D(p)\not =0.$ Then denoting $a:=p/D(p)$ we have $D(a)=1$ (such an element $a$ is called a slice for $D$).
The ring $B$ is contained in the localization $B[c^{-1}],$ where $c:=D(p)$ and the derivation $D$ is locally nilpotent on $B[c^{-1}].$ Note that $B[c^{-1}]\subseteq F_1,$ where $F_1$ is the field of constants for $L_1=FI+FD$ in $R.$ Besides,  by Principle 11 from \cite{Fr}  it holds $B[c^{-1}]=B_{0}[a],$ where $B_0$ is the kernel of $D$ in $B[c^{-1}].$ This completes the proof because $B\subseteq B[c^{-1}]$ and every element $\overline D$ of $FI$ is of the form $\overline{D}=b_1D_1+\ldots b_mD_{m}, b_i\in B.$

\end{proof}

\begin{lemma}\label{factor}
	Let $L$ be a nilpotent subalgebra of the Lie algebra
	$W(A)$, $Z = Z(L)$ the center of $L,$
	$F$ the field of constants of $L$ in $R$
	and $I = RZ \cap L$. Let $\rk _{R}Z=n-2.$
	Then the following statements for the Lie algebra $FL/FI$ hold:
	\begin{enumerate}
	\item if $FL / FI$ is  abelian, then
	$\dim_F FL / FI = 2$;
	\item if $FL / FI$ is nonabelian, then there exist
	 elements $D_{n-1}, D_n \in FL, \; b \in R$ such that
	$$
	FL / FI = F \langle  D_{n-1} + FI,\, bD_{n-1} + FI,\,		\dots,\, \frac{b^k}{k!} D_{n-1} + FI,\, D_n + FI \rangle
	$$
	with $ k\geq 1, D_n(b) = 1, \; D_{n-1}(b) = 0, \; D(b) = 0$
	for all $D \in FI$.
	\end{enumerate}
\end{lemma}
\begin{proof}
Let us choose a basis $D_1,\dots, D_{n-2}$ of the center
$Z$ over the field $R$ and any central ideal  $FD_{n-1} + FI$ of the quotient algebra
$FL/FI.$ Denote the intersection
$R(I + \mathbb K D_{n-1}) \cap L$ by $I_1$.
Then it is easy to see that $FI_1$ is an ideal of the Lie algebra $FL$ of rank $n-1$ over $R$ and the Lie algebra $FL/FI_1$ is of dimension $1$  over $F$ (by Lemma 5 from \cite{MP1}).
Let us choose an arbitrary element $D_n \in FL\setminus FI_1.$ Then
$D_1,\dots,D_n$ is a basis of the Lie algebra $FL$ over the field $R$.

\underline{Case 1.} The quotient algebra $FL / FI$ is abelian.
Let us show that
$FL/FI = F\langle D_{n-1} + FI,\, D_n + FI\rangle .$
Indeed, take any elements $S_1 + FI,\; S_2 + FI$ of
$FL/FI$ and write
$$
S_1 = \sum_{i=1}^n r_i D_i, \quad
S_2 = \sum_{i=1}^n s_i D_i, \quad
r_i, \; s_i \in R, \; i,j = 1,\dots,n.
$$
From the equalities
$[D_i, S_1] = [D_i, S_2] = 0$ $i=1,\dots,n-2$
(recall that $D_i \in Z(L),\; i=1,\dots,n-2$) it follows that
\begin{equation}\label{formula1}
D_i(r_j) = D_i(s_j) = 0,\;i=1\dots,n-2,\; j=1,\dots,n.
\end{equation}
Since $[FL, FI]\subseteq FI$  we have
$[D_i, S_1], [D_i, S_2] \in FI$ for $i=n-1,\, n.$
Taking into account the  equalities (\ref{formula1}) we derive that
$$D_i(s_j) = D_i(r_j) = 0,\; i=n-1,\,n,\;j=n-1,\,n.$$
Therefore it holds $s_i, r_i \in F$ for $i=n-1,\,n$ and
the elements   $D_{n-1} + FI,\, D_n + FI$ form
a basis for the abelian Lie algebra $FL/FI$ over the field $F.$

\underline{Case 2.} $FL / FI$ is nonabelian.  Then
$\dim_F FL/FI \ge 3$ because the Lie algebra $FL/FI$ is nilpotent. Let us show that the ideal
$FI_1/FI$ of the Lie algebra $FL/FI$ is abelian
(recall that $I_1 = R(I + \mathbb K D_{n-1}) \cap L$).
Since $D_{n-1} + FI$ lies in the center of the quotient algebra $FL/FI$ we have   for any element $r D_{n-1} + FI$ of
the ideal $FI_1 / FI$ the following equality:
$
[D_{n-1} + FI, rD_{n-1} + FI]=FI.
$
Hence $D_{n-1}(r)D_{n-1} + FI = FI$.
The last equality implies $D_{n-1}(r) = 0$.
But then  for any elements  $rD_{n-1} + FI,\,sD_{n-1} + FI$
of $FI_1 / FI$ we get
$$
[rD_{n-1} + FI,sD_{n-1} + FI] = [rD_{n-1},sD_{n-1} + FI =
(D_{n-1}(s)r - sD_{n-1}(r))D_{n-1}+FI=FI.
$$
The latter means that $FI_1/FI$ is an abelian ideal of $FL/FI$.

Further, the nilpotent linear operator $\ad D_n$ acts
on the linear space $FI_1 / FI$ with
$\Ker(\ad D_n) = F D_{n-1} + FI$. Indeed, let
$\ad D_n(rD_{n-1} + FI)=FI$. Then
$[D_n, rD_{n-1}] \in FI$ and therefore $D_n(r)D_{n-1} \in FI$.
   This relation implies $D_n(r) = 0$ and taking into account the equalities $D_i(r) = 0,\; i=1,\dots,n-1$
we get that $r \in F$ and
$\Ker(\ad D_n) = F D_{n-1} + FI$.
It follows from this relation  that the linear operator  $\ad D_n$ on $FI / FI_1$  has only one Jordan chain and the Jordan basis can be chosen with  the  first element $D_{n-1} + FI.$
Since $\dim FI_1 / FI \ge 2$ (recall that $\dim_F FL / FI \ge 3$)
the chain is of length $\geq 2.$
Let us take the second element of the Jordan chain in the form
$b D_{n-1} + FI,\; b \in R$. Then
$\ad D_n (bD_{n-1} + FI) = D_{n-1} + FI$ and hence
 $D_n(b) = 1$. The inclusion $[D_{n-1}, bD_{n-1}]\in FI$ implies the equality $D_{n-1}(b) = 0$, and analogously one can obtain $D_i(b) = 0,\; i=1,\dots,n-2$.

 If $\dim FI_1 / FI \ge 3$
and $cD_{n-1} + FI$ is the third element of the Jordan chain of $\ad D_n$, then repeating the above considerations we get $D_n(c) = b$.
Then  the element $\alpha = \frac{b^2}{2!} - c \in R$ satisfies the relations
$D_{n-1}(\alpha) = D_n(\alpha) = 0$ and
$D_i(\alpha) = 0,\; i=1,\dots,n-2$ since
$D_i(b) = D_i(c) = 0$. Therefore,
$\alpha = \frac{b^2}{2!} - c \in F$ and
$c = \frac{b^2}{2!} + \alpha$.
Since $\alpha D_{n-1}+FI \in \Ker(\ad D_n)$ we can take
the third element of the Jordan chain in the form $\frac{b^2}{2!}D_{n-1} + FI.$ Repeating the consideration  one can build the needed  basis of the Lie algebra $FL/FI.$
\end{proof}

\begin{lemma}\label{twovariables}
	Let $L$ be a nilpotent subalgebra of
	$W(A)$ with the center  $Z = Z(L)$ of $\rk _{R}Z=n-2,$
	$F$ the field of constants for $L$ in $R$
	and $I = RZ \cap L$.
	If $S,T $ are  elements of $L$ such that
	$[S,T] \in I$, the rank of the  subalgebra
	$L_1$ spanned by $I, S, T$ equals $n$ and
	$C_{FL}(FI) = FI$, then there exist elements
	$a,b \in R$ such that
	$S(a) = 1,\, T(a) = 0,\, S(b) = 0,\, T(b) = 1$ and
	$D(a) = D(b) = 0$ for each $D \in I$. Besides, every element $D\in FI$ can be written in the form $D=f_1(a, b)D_1+\cdots +f_{n-2}(a, b)D_{n-2}$ with some polynomials $f_{i}(u, v)\in F[u, v].$
\end{lemma}
\begin{proof}
	Let us choose a basis $D_1,\dots,D_{n-2}$ of
	$Z$ over $R$. By the lemma conditions,
	$D_1,\dots,D_{n-2},S,T$ is a basis of $L$ over $R$.
    The ideal $FI$ of the Lie algebra $FL$ is abelian by Lemma \ref{central}
    and  $\ad S, \ad T$ are commuting  linear operators on the vector space $FI$  (over $F$). Take a basis $T_1, \ldots , T_s$ of $FI$ over $F$ (recall that $\dim _{F}FL<\infty$ by Theorem 1 from \cite{MP1}) and write $T_i=\sum _{j=1}^{n-2}r_{ij}D_j$ for some $r_{ij}\in R, i=1, \ldots ,s, j=1, \ldots ,n-2.$ Denote by
    $$B=F[ r_{ij}, i=1, \ldots ,s, j=1, \ldots ,n-2],$$
     the subring of $R$ generated by $F$ and all the coefficients  $r_{ij}.$ Then $B$ is invariant under the derivations $S$ and $T,$  these derivations are locally nilpotent on $B$ and linearly independent over $R$ (by the condition $C_{FL}(FI)=FI$ of the lemma). By Lemma \ref{elements}, there exists an element $a\in B[c^{-1}]$ such that $$S(a)=1, D_{i}(a)=0, i=1, \ldots , n-2$$
      (here $c=S(p)$ for a preslice $p$ for $S$ in $B$). Since $c\in \Ker S$ and $[S, T]=0$ one can assume without loss of generality that  $T(c)\in \Ker T.$ But then $T$ is a locally nilpotent derivation on the subring $B[c^{-1}].$ Repeating these considerations we can find an element $b\in B[c^{-1}][d^{-1}]$  with $T(b)=1$ (here $d$ is a preslice for the derivation $T$ in $B[c^{-1}]).$ Denote $B_1=B[c^{-1}, d^{-1}],$ the subring of $R$ generated by $B, c^{-1}, d^{-1}.$  Then using standard facts about locally nilpotent derivations (see, for example Principle 11 in \cite{Fr}) one can show that $B_1=B_{0}[a, b],$ where $B_0=\Ker S\cap \Ker T.$ Therefore every element $h$ of $B_1$ can be written in the form $h=f(a, b)$ with $f(u, v)\in F[u, v]$ (note that
     $$F=\Ker T\cap \Ker S\cap_{i=1}^{n-2}\Ker D_{i}.$$
      It follows from this representation of elements of $B_1$ that every element of the ideal $FI$ can be written in the form
    $$D=f_1(a, b)D_1+\cdots +f_{n-2}(a, b)D_{n-2}$$
     with some polynomials $f_{i}(u, v)\in F[u, v].$

\end{proof}

\section{The main results}

\begin{theorem} \label{classification}
	Let $L$ be a nilpotent subalgebra of rank $n\geq 3$ over $R$ from the Lie algebra $W(A)$, $Z = Z(L)$ the center of L with $\rk _{R}Z\geq n-2,$
	$F$ the field of constants of $L$ in $R$.
	Then one of the following statements holds:

1) $\dim_F FL = n$ and $FL$ is either abelian or is a direct sum of a nonabelian nilpotent Lie algebra of dimension $3$ and an abelian Lie algebra;

2) $\dim _{F}FL\geq n+1$ and $FL$  lies in one of the locally
	nilpotent subalgebras $L_1, L_2$ of $W(A)$ of rank $n$ over $R$,
	which have a basis $D_1, \dots, D_n$ over $R$ satisfying the relations $[D_i, D_j] = 0,\; i,j=1,\dots,n$ and are one  of the form:

	$
	L_1 = F\langle
	\{ \frac{b^i}{i!} D_1 \}_{i=0}^{\infty}, \dots ,
	\{ \frac{b^i}{i!} D_{n-1} \} _{i=0}^{\infty},
	D_n
	\rangle
	$

	for some $b \in R$ such that $D_i(b) = 0,\; i=1,\dots, n-1$ and
	$D_n(b) = 1$,

	$
	L_2 = F\langle
	\{ \frac{a^i b^j}{i!j!} D_1 \}_{i,j=0}^{\infty}, \dots ,
	\{ \frac{a^i b^j}{i!j!} D_{n-2} \}_{i,j=0}^{\infty},
	\{ \frac{b^i}{i!} D_{n-1} \}_{i=0}^{\infty},
	D_n
	\rangle
	$

	for some $a, b \in R$ such that
	$D_{n-1}(a) = 1,\; D_n(a) = 0,\; D_{n-1}(b) = 0, D_n(b) = 1$,
	$D_i(a) = D_i(b) = 0, \; i=1,\dots,n-2$.
\end{theorem}
\begin{proof}
By Lemma \ref{central}, $I = RZ \cap L$ is an abelian ideal of $L$ and therefore $FI$ is an abelian ideal of the Lie algebra $FL$ (here the Lie algebra $FL$ is considered over the field $F$). Let $\dim _{F}FL=n.$ It is obvious that  $\dim _{F}M=\rk _{R}M$ for any subalgebra $M$ of the Lie algebra $FL,$ in particular $\dim _{F}FZ\geq n-2$ because of conditions of the theorem. We may restrict ourselves only on nonabelian algebras and assume  $\dim _{F}FZ=n-2$ (in case $\dim _{F}FZ\geq n-1$ the Lie algebra $FL$ is abelian).  Since $FL$ is nilpotent of nilpotency class $2$ one can easily show that $FL$ is a direct sum of a nonabelian Lie algebra of dimension $3$ and an abelian algebra and satisfies the condition 1) of the theorem. So we may assume further that $\dim _{F}FL\geq n+1.$

\underline{Case 1.} $\rk _{R}Z=n-1.$ Then $FI$ is of codimension $1$ in $FL$ by Lemma 5 from \cite{MP1}. Therefore $\dim _{F}FI\geq n$ because of  $\dim _{F}FL\geq n+1$ and $\dim _{F}FL/FI=1.$ We obtain the strong inclusion $FZ\varsubsetneqq FI$ because of $\dim _{F}FZ=n-1.$
 Take a basis $D_1, \ldots , D_{n-1}$ of $Z$ over $R$ and an element $D_{n}\in FL\setminus FI.$ Then $D_1, \ldots , D_{n}$ is a basis for $FL$ over $R$ and $[D_n, FI]\not = 0.$ Using Lemma \ref{elements} one can easily show that $FL$ is contained in a subalgebra of type $L_1$  from $W(A).$

\underline{Case 2.} $\rk _{R}Z=n-2$ and $\dim_F FI = n-2$.
Then $FI=FZ,$  $\dim _{F}FL/FI\geq 3$ and therefore by Lemma  \ref{factor} the quotient algebra $FL/FI$ is  of the form
\begin{equation}\label{mult}
FL/FI = F\langle \{ \frac{b^i}{i!}D_{n-1} + FI \}_{i=0}^k,
D_n + FI \rangle
\end{equation}
for  some $k\geq 1,$ $b \in R$ such that $D_n(b) = 1,\; D_{n-1}(b) = 0$ and
$D(b) = 0$ for each $D \in FI$.

The  $F$-space
$$
J = F\langle
	\{ \frac{b^i}{i!} D_1  \}_{i=0}^{\infty}, \dots,
	 \{ \frac{b^i}{i!} D_{n-1} \}_{i=0}^{\infty}
	\rangle
$$
is an abelian subalgebra of $W(A)$ and $[FL, J] \subseteq J$.
Therefore the   sum
$$
J + F\langle
\{ \frac{b^i}{i!} D_{n-1} \}_{i=0}^{\infty},
D_n \rangle
$$
 is a subalgebra of the Lie algebra $W(A)$. If $[D_n, D_{n-1}] \ne 0,$ then taking
into account the relation $[D_n, D_{n-1}] \in FI$ one can write
$$[D_n, D_{n-1}] = \alpha _1 D_1 + \dots + \alpha _{n-2} D_{n-2}$$ for some $ \alpha _i\in F$ (recall that $FI=FZ$).
Consider the element of $W(A)$ of the form
$$\widetilde{D}_{n-1} =
D_{n-1} - \alpha _1 bD_1 - \dots - \alpha _{n-2} bD_{n-2}.$$
Since   $[D_n, \widetilde{D}_{n-1}] = 0,\; \widetilde{D}_{n-1}(b) = 0$
one can replace the element  $D_{n-1}$ with the element $\widetilde{D}_{n-1}$ and assume without loss of generality that $[D_n , D_{n-1}]=0.$  As a result we get the Lie algebra of the type $L_1$ from
 the statement of the theorem.

\underline{Case 3.} $\rk _{R}Z=n-2$ and $\dim_F FI>n-2.$
First, suppose $C_{FL}(FI) = FI$. Then by Lemma \ref{twovariables} there are a basis
$D_1,\dots,D_{n-2}$ of the ideal $FI$ over $R$ and  elements $a,b \in R$ such
that $$D_{n-1}(a) = 1,\; D_n(a) = 0,\; D_{n-1}(b) = 0,\; D_n(b) = 1$$ and
$$D_i(a) = D_i(b) = 0,\; i=1,\dots, n-2$$
 and each element
$D \in FI$ can be written in the form
$$
D = f_1(a,b) D_1 + \dots + f_{n-2}(a,b) D_{n-2}
$$
for some polynomials $f_i(u,v) \in F[u,v]$.

Consider the $F$-subspace
$$J = F[a,b] D_1 + \dots + F[a,b] D_{n-2}$$
of the Lie algebra $W(A).$
It is easy to see that $J$ is an abelian subalgebra of $W(A)$ and
$[FL, J] \subseteq J$. If $[D_n, D_{n-1}] = 0$, then it is obvious that
the subalgebra $FL + J$ is of type $L_2$ of the theorem and $FL\subset L_1.$
Let $[D_n, D_{n-1}] \ne 0$. Since $[D_n, D_{n-1}] \in FI$, it follows
$$
[D_n, D_{n-1}] = h_1(a,b) D_1 + \dots + h_{n-2} D_{n-2}
$$
for some polynomials $h_i(u,v) \in F[u,v]$. Then the subalgebra $J$
has such an element
$$T = u_1(a,b) D_1 + \dots u_{n-2}(a,b) D_{n-2}$$
that $D_n(u_i(a,b)) = h_i(a,b), i=1, \ldots , n-2$ (recall that $D_n(a) = 0,\; D_n(b) = 1$),
and hence the element $\widetilde{D}_{n-1}=D_{n-1} - T$ satisfies the equality
$[D_n, T] = 0$. Replacing $D_{n-1}$ with $\widetilde{D}_{n-1}$ we get
the needed  basis of the Lie algebra $FL + J$ and see that $FL$ can be embedded into the Lie $L_2$ of $W(A).$
So in case of $C_{FL}(FI) = FI$ the Lie algebra $FL$
can be isomorphically embedded into the Lie algebra of type $L_2$ from the statement of the theorem.

Further, suppose $C_{FL}(FI) \ne FI$. Since $C_{FL}(FI) \supseteq FI$ one can easily show that $D_{n-1} \in C_{FL}(FI) \setminus FI$
(note that $FL / FI$ has the unique minimal ideal $FD_{n-1} + FI$).
Then $[D_{n-1}, FI] = 0$, and therefore $[D_n, FI]\not =0.$ Therefore by Lemma \ref{elements} there is an element $c \in R$ such  that
$$D_n(c) = 1,\; D_{n-1}(c) = 0,\; D_i(c) = 0, \  i=1,\dots,n-2.$$
Moreover, each element of $FI$ is of the form
$g_1(c) D_1 + \dots + g_{n-2}(c) D_{n-2}$ for some polynomials
$g_i(u) \in F[u]$. By Lemma \ref{factor}, the quotient algebra $FL/FI$ is of the form
$$
FL/FI=F\langle \{ \frac{b^i}{i!}D_{n-1} + FI \}_{i=0}^k,
D_n + FI \rangle
$$
for some $b \in R, k\geq 1$ such that $D_n(b) = 1,\; D_{n-1}(b) = 0$.
But then
 $$D_{n-1}(b-c) = 0,\; D_n(b-c) = 0,\; D_i(b-c) = 0,$$ and hence
$b-c = \alpha$ for some $\alpha \in F$. Without loss of generality
we can assume $b = c$.
The locally nilpotent subalgebra

$$
L_1 =
	F\langle
	\{ \frac{a^i b^j}{i!j!} D_1\}_{i,j=0}^{\infty}, \dots
	\{ \frac{a^i b^j}{i!j!} D_{n-2}\}_{i,j=0}^{\infty},
	\{ \frac{b^i}{i!} D_{n-1} \}_{i=0}^{\infty},
	D_n
	\rangle
$$

of the Lie algebra $W(A)$ contains $FL$ and satisfies the conditions for the Lie algebra of type $L_2$ from the statement of the theorem, possibly except the condition $[D_n, D_{n-1}] = 0$.
If $[D_n, D_{n-1}] \ne 0$, then from the inclusion
$[D_n, D_{n-1}] \in FI$ it follows that
$$
[D_n, D_{n-1}] = f_1(b)D_1 + \cdots + f_{n-2}(b)D_{n-2}
$$
for some polynomials $f_i(u) \in F[u]$.

One can easily show that there is an element
$$\overline{D} = h_1(b) D_1 + \dots + h_{n-2}(b) D_{n-2} \in L_1,$$ such that $[D_n, \overline{D}] = [D_n, D_{n-1}]$ (one can take antiderivations $h_i$ for polynomials $f_i, i=1, \ldots , n-2$). Replacing $D_{n-1}$ with
$D_{n-1} - \overline{D}$ we get the needed basis over $R$ of the Lie algebra $L_2$.
\end{proof}

\begin{remark}
Any  Lie algebra of dimension $n$   over $F$ can be realized as a Lie algebra of rank $n$ over $R$ by Theorem 2 from \cite{Mak1}. So the  Lie  algebra of type 1) from Theorem 1 can be chosen in any way possible.
\end{remark}

As a corollary we get the next statement about embedding of Lie algebras of derivations.

\begin{theorem}\label{embed}
	Let $L$ be a nilpotent subalgebra of rank $n$ over $R$  of the Lie algebra
	$W(A)$, $Z = Z(L)$ the center of $L$   and
	$F$ the field of constants of $L$ in $R.$ If  $\rk _{R}Z\geq n-2,$ 	then the Lie algebra $FL$ can be isomorphically embedded (as an abstract Lie algebra) into the triangular Lie algebra $u_n(F).$
\end{theorem}
\begin{proof}
First, suppose $\dim _{F}FL = n$. If $FL$ is abelian, then $FL$ is isomorphically embeddable into the Lie algebra $u_n(F)$ because the subalgebra $F\langle \frac{\partial}{\partial x_1}, \ldots , \frac{\partial}{\partial x_n}\rangle$ of $u_n(F)$ is abelian of of dimension  $n$ over $F.$ So one can assume that $FL$ is nonabelian. Then by Theorem \ref{classification} $FL=M_1\oplus M_2,$ where $M_1$ is an abelian Lie algebra of dimension $n-3$ over $F$ and $M_2$ is nilpotent nonabelian with $\dim _{F}M_2=3.$  The subalgebra $H_2=F\langle \frac{\partial}{\partial x_1},
			    \frac{\partial}{\partial x_2} +
			    x_3 \frac{\partial}{\partial x_1},
			    \frac{\partial}{\partial x_3} \rangle$
of the Lie algebra $u_n(F)$ is obviously isomorphic to $M_2.$ The abelian subalgebra $H_1=F\langle \frac{\partial}{\partial x_4}, \dots ,
		\frac{\partial}{\partial x_n} \rangle , n\geq 4$  is isomorphic to the Lie algebra $M_1.$ So $FL\simeq H_1\oplus H_2$ is isomorphic to a subalgebra of $u_n(F).$  Note that $H_1\oplus H_2$ is of rank $n$ over the field $\mathbb K(x_1, \ldots , x_n)$ of rational functions in $n$ variables.

Next, let  $\dim _{F}FL>n.$  By Theorem \ref{classification},
the Lie algebra $FL$ lies in one of the subalgebras of types $L_1$ or $L_2.$
Therefore it is sufficient to show that the subalgebras $L_1, L_2$ of $W(A)$ from Theorem  \ref{classification}  can be isomorphically embedded into the Lie algebra $u_n(F)$.
In case  $L_1,$ we define a mapping $\varphi$  on the basis $D_1,\dots,D_n, \{\frac{b^i}{i!} D_i\}_{i=1}^{\infty}$  of $L_1$ over $R$
by the rule $\varphi(D_i) = \frac{\partial}{\partial x_i}, i=1, \ldots ,n ,$
$\varphi(\frac{b^i}{i!} D_i) =
\frac{x_n^i}{i!}\frac{\partial}{\partial x_i}, i=1, \ldots , n-1$ and then extend it on $L_1$ by linearity.
One can  easy to see that the mapping $\varphi$ is an isomorphic embedding
of the Lie algebra $L_1$ into $u_n(F)$. Analogously, on $L_2$ we define a mapping $\psi: L_2 \to u_n(F)$ by the rule
$$
\psi(D_i) = \frac{\partial}{\partial x_i},\; i=1,\ldots, n, \
\psi(\frac{a^ib^j}{i!j!}D_k) = \frac{x_{n-1}^i x_n^j}{i!j!}\frac{\partial}{\partial x_k}, k=1,\dots,n-2
$$
$$
 \psi(\frac{b^i}{i!}D_{n-1}) = \frac{x_n^i}{i!}
\frac{\partial}{\partial x_{n-1}},  \ i\geq 1, j\geq 1,  \ \ \
$$
and further by linearity.
Then $\psi$ is an isomorphic embedding of the Lie algebra $L_2$ into the Lie algebra $u_n(F)$.
\end{proof}


%

\begin{thebibliography}{99}
\bibitem{Bav1}
 V.V. Bavula. Lie algebras of triangular polynomial derivations and an isomorphism criterion for their Lie factor algebras.
Izv. RAN. Ser. Mat., 2013, {\bf 77}, Issue 6, 3--44.

\bibitem{Bav2} V.V.Bavula, Every monomorphism of the Lie algebra of triangular polynomial derivations is an automorphism, Comptes Rendus Mathematique, (2012), v.350, no.11-12, p.553-556.

\bibitem{BP}  V.M.Bondarenko, A.P.Petravchuk, Wildness of the problem of classifying nilpotent Lie algebras of vector fields in four variables, Linear algebra and its applications, v.568, (2019), P.165-172.

 \bibitem{Fr}  G. Freudenburg, Algebraic theory of locally nilpotent derivations, Encyclopaedia of Math. Sciences, 136, 2006.

 \bibitem{Mak1} Ie.Makedonskyi, On noncommutative bases of the free module $W_n(K)$, Communications in Algebra, v.44, issue 1 (2016), 11-25.


 \bibitem{MP1}
Ie. O. Makedonskyi, A.P.  Petravchuk. On   nilpotent and solvable Lie algebras of  derivations. Journal of Algebra, 2014, {\bf 401}, 245--257.

\bibitem{Nowicki}
A. Nowicki. Polynomial Derivations and their Rings of Constants. Uniwersytet Mikolaja Kopernika, Torun. 1994.


\bibitem{PS} A.P. Petravchuk, K.Ya. Sysak, Lie algebras consisting of locally nilpotent derivations, Journal of Lie theory, (2017), v.27, no.4, 1057-1068.


\bibitem{Petr} Petravchuk A.P. , On  nilpotent Lie algebras of derivations of fraction fields, Algebra and Discrete Math., (2016), v.22, no. 1, 118-131

\bibitem{Sysak} Sysak K.~Ya.,  On nilpotent Lie algebras of derivations with large center, Algebra and Discrete Mathematics. (2016), v.21, no.1, 153-162.

\end{thebibliography}
\end{document}